\documentclass[11pt]{article}
\usepackage[mathscr]{eucal}
\usepackage{amsmath}
\usepackage{epsfig,epsf,psfrag}
\usepackage{amssymb,amsfonts,latexsym}
\usepackage{amsmath}
\usepackage{graphicx}
\usepackage{epstopdf}
\usepackage{bm,xcolor,url}
\usepackage{fixltx2e}
\usepackage{array}
\usepackage{verbatim}
\usepackage{algpseudocode, cite}
\usepackage{algorithm}
\usepackage{verbatim}
\usepackage{textcomp}
\usepackage{mathrsfs}
\usepackage[referable]{threeparttablex}
\usepackage{subfig}
\usepackage{amsthm}
\usepackage{enumerate}
\usepackage{footnote}
\usepackage{enumitem}
\usepackage{xr}
\usepackage{mathtools}
\catcode`~=11 \def\UrlSpecials{\do\~{\kern -.15em\lower .7ex\hbox{~}\kern .04em}} \catcode`~=13 

\allowdisplaybreaks[3]

\newcommand{\tnorm}[1]{{\left\vert\kern-0.25ex\left\vert\kern-0.25ex\left\vert #1 
    \right\vert\kern-0.25ex\right\vert\kern-0.25ex\right\vert}}
\newcommand{\tnormt}[1]{{\vert\kern-0.25ex\vert\kern-0.25ex\vert #1 
    \vert\kern-0.25ex\vert\kern-0.25ex\vert}}

\newcommand{\andd}{\;\; \mbox{and}\;\;}
\newcommand{\ipt}{\lranglet}

\newcommand{\clconv}{\mathsf{clconv}\,}

\newcommand{\norm}[1]{\left\Vert#1\right\Vert}
\newcommand{\normt}[1]{\Vert#1\Vert}
\newcommand{\abst}[1]{\vert#1\vert}

\newcommand{\nn}{\nonumber}

\newcommand{\cl}{\mathsf{cl}\,}

\newcommand{\ran}{\mathsf{ran}\,}

\newcommand{\conv}{\mathsf{conv}\,}

\newcommand{\calA}{\mathcal{A}}
\newcommand{\calB}{\mathcal{B}}
\newcommand{\calC}{\mathcal{C}}
\newcommand{\calD}{\mathcal{D}}
\newcommand{\calE}{\mathcal{E}}

\newcommand{\calH}{\mathcal{H}}

\newcommand{\calK}{\mathcal{K}}

\newcommand{\calO}{\mathcal{O}}
\newcommand{\calP}{\mathcal{P}}
\newcommand{\calQ}{\mathcal{Q}}

\newcommand{\calS}{\mathcal{S}}

\newcommand{\bbR}{\mathbb{R}}
\newcommand{\bbS}{\mathbb{S}}

\newcommand{\bbV}{\mathbb{V}}
\newcommand{\bbW}{\mathbb{W}}

\DeclareMathAlphabet{\mathbsf}{OT1}{cmss}{bx}{n}

\newcommand{\lranglet}[2]{\langle{#1},{#2}\rangle}

\DeclareMathOperator{\st}{\;\;s.t.\;\;}

\newtheorem{theorem}{Theorem} 
\newtheorem{lemma}{Lemma}
\newtheorem{prop}{Proposition}
\newtheorem{corollary}{Corollary}
\theoremstyle{definition}
\newtheorem{remark}{Remark}
\newtheorem{example}{Example} 
\newtheorem{definition}{Definition} 
\newcommand{\qednew}{\nobreak \ifvmode \relax \else
      \ifdim\lastskip<1.5em \hskip-\lastskip
      \hskip1.5em plus0em minus0.5em \fi \nobreak
      \vrule height0.75em width0.5em depth0.25em\fi}

\usepackage{fullpage}
\numberwithin{equation}{section}
\numberwithin{theorem}{section}
\numberwithin{corollary}{section}
\numberwithin{lemma}{section}
\numberwithin{example}{section}
\numberwithin{remark}{section}
\numberwithin{definition}{section}
\numberwithin{prop}{section}
\usepackage[colorlinks=true,breaklinks=true,bookmarks=true,urlcolor=blue,
     citecolor=blue,linkcolor=blue,bookmarksopen=false,draft=false]{hyperref}

\usepackage{xspace}

\newcommand{\Rnd}{\bbR^n_\downarrow}

\newcommand{\Rndp}{(\bbR^n_+)_\downarrow}

\usepackage[numbers]{natbib}

\expandafter\def\expandafter\normalsize\expandafter{%
    \normalsize%
    \setlength\abovedisplayskip{8pt}%
    \setlength\belowdisplayskip{8pt}%
    \setlength\abovedisplayshortskip{-8pt}%
    \setlength\belowdisplayshortskip{2pt}%
}

\begin{document}

\title{On the Convexification of Spectral Sets \\Induced by Non-Invariant Sets
}

\date{}


\author{Renbo Zhao \thanks{Department of Business Analytics, Tippie College of Business, University of Iowa (\href{mailto:renbo-zhao@uiowa.edu}{renbo-zhao@uiowa.edu}).} 
}


\vspace{-5cm}

\maketitle

\begin{abstract}
Given  a  finite-dimensional {FTvN system} $(\mathbb{V},\mathbb{W},\lambda)$~\cite{Gowda_19}, we study the convexification of the spectral set $\lambda^{-1}(\mathcal{C})$ induced by a set $\mathcal{C} \subseteq \mathbb{W}$.   While the case of invariant $\mathcal{C}$ has been relatively well-studied, the results for non-invariant $\mathcal{C}$ are largely lacking  in the literature. We fill this void by developing simple and geometric characterizations of the convex hull and closed convex hull of $\lambda^{-1}(\mathcal{C})$ when $\mathcal{C}$ has no  invariance property. We further specialize our results to the case of invariant $\mathcal{C}$, and  obtain new convexifications of $\lambda^{-1}(\mathcal{C})$ in this case. 
\end{abstract}

\section{Introduction}\label{sec:intro}

This work is devoted to study the convexification of 
spectral sets 
defined in  the (finite-dimensional) FTvN system~\cite{Gowda_19}. We start 
with the basic definitions of the FTvN system. 
Let $\bbV$ and $\bbW$ be two finite-dimensional real inner-product spaces with a (nonlinear) map $\lambda:\bbV\to \bbW$. We call $(\bbV,\bbW,\lambda)$ a {\em FTvN system} 
 if it satisfies the following two conditions: 
\begin{enumerate}[label=(P\arabic*),leftmargin=1.5\parindent]
\item For all $x,y\in\bbV$, we have $\ipt{x}{y}\le \ipt{\lambda(x)}{\lambda(y)}$. 
\label{P1}
\item \label{P2}  For all  $c\in\bbV$ and $u\in\lambda(\bbV)$, there exists $x\in\bbV$ such that $\lambda(x) = u$ and $\textstyle\ipt{c}{x}= \ipt{\lambda(c)}{\lambda(x)}.$
\end{enumerate}\vspace{-0.5ex}
In fact, from~\ref{P1} and~\ref{P2}, we can easily 
obtain 
two other useful properties of $\lambda$. 
First, the range of $\lambda$  (denoted by $\ran\lambda:=\lambda(\bbV)$) is a nonempty closed convex cone, which we denote by $\calK\subseteq \bbW$ (cf.~\cite[Corollary~2.8]{Gowda_19}). Second, $\lambda$ is {\em norm-preserving}, namely $\norm{\lambda(x)} = \norm{x}$ for all $x\in\bbV$, where the norm $\normt{\cdot}$ is induced by the inner product on $\bbV$ or $\bbW$ (cf.~\cite[Section~3]{Gowda_23}).   

The FTvN system is a novel and elegant framework that includes numerous examples.
Perhaps the simplest non-trivial example is that $\bbV = \bbW = \bbR^n$ (equipped with standard scalar product) and $\lambda$ is the {\em reordering map}, namely $\lambda(x):= x^\downarrow$ for $x\in\bbR^n$, where $x^\downarrow$ is obtained by rearranging entries of $x$ in non-increasing order. 
Two other prominent examples include
\begin{enumerate}[label=(E\arabic*),leftmargin=1.5\parindent,parsep=1pt,topsep=5pt]
\item \label{eg:S^n} $\bbV = \bbS^n$ (i.e., the space of  $n \times n$ real symmetric matrices with trace inner
product), $\bbW = \bbR^n$  and $\lambda$ is the {\em eigenvalue 
map}, namely for $X\in \bbS^n$, $\lambda(X)$ 
denotes the vector of eigenvalues of $X$ written in non-increasing order, and 
\item  \label{eg:R^m_n} $\bbV = \bbR^{m\times n}$ (equipped with trace inner product), $\bbW = \bbR^{d}$ (where $d:=\min\{m,n\}$) and $\lambda$ is the {\em singular-value map}, i.e.,  for $X\in \bbR^{n\times n}$, $\lambda(X)$ 
denotes the vector of singular values of $X$ written in non-increasing order.
\end{enumerate}
In fact, the FTvN system subsumes two important and broad  
systems, namely 
the {\em normal decomposition system}~\cite{Lewis_96} and the system induced by {\em complete and isometric  hyperbolic polynomials}~\cite{Bauschke_01a}, which include the examples above as special cases.  For details and more examples, we refer readers to~\cite[Section~4]{Gowda_23}.  For convenience, we shall call $\lambda:\bbV\to \bbW$ the {\em spectral map}, which emphasizes the important role of the example~\ref{eg:S^n} above.  
 

In this work, we 
study the convexification of the 
 {\em spectral sets} in the {FTvN system}  $(\bbV,\bbW,\lambda)$. 
 A 
 set $\calE\subseteq \bbV$ is called {\em spectral} in $(\bbV,\bbW,\lambda)$ if there exists 
 $\calC\subseteq \bbW$ such that 
 \begin{equation}
 \calE := \lambda^{-1}(\calC)= \lambda^{-1}(\calC\cap\calK).
 \end{equation}
 Since $\lambda(\calE) = \calC\cap\calK$, it is clear that $\calE$ is non-empty if and only if $\calC\subseteq \bbW$ is {\em feasible}, i.e., $\calC\cap\calK\ne \emptyset$. For non-triviality, we shall always focus on the feasible $\calC$ in the sequel. 

The spectral set $\calE$ 
appears as the spectral constraints in many optimization problems, where the spectral map $\lambda$ takes various forms, including the absolute-reordering map on $\bbR^n$~\cite{Kim_22}, the eigenvalue map on $\bbS^n$~\cite{Garner_23,Li_23} or more generally on a Euclidean Jordan algebra of rank $n$~\cite{Ito_23}, and the singular-value map on $\bbR^{m\times n}$~\cite{Li_23}. Recently, some works~\cite{Gowda_19,Ito_23} 
pioneered the study of the optimization problems with spectral constraints $\calE$ 
defined by the spectral map in the FTvN system, which unifies and generalizes all the aforementioned forms of $\lambda$. 

Despite the important role that the spectral set $\lambda^{-1}(\calC)$ plays in optimization, to our knowledge, the convexification results of $\calE$ 
has 
been largely limited to the setting where $\calC$ possesses certain {\em invariance} properties~\cite{Lewis_96,Kim_22,Jeong_23}. Indeed, the proper notion of invariance 
depends on the spectral map $\lambda$. For example, if $\lambda$ is the eigenvalue (resp.\ singular-value) map, then the corresponding notion of invariance is the permutation-  (resp.\ permutation- and sign-) invariance. More generally, 
in the FTvN system,  the invariance is defined 
via the notion of the {\em reduced system} of $(\bbV,\bbW,\lambda)$~\cite{Gowda_23}, which we will elaborate in 
Section~\ref{sec:prelim}.  
Indeed, under different contexts,  many previous works~\cite{Lewis_96,Kim_22,Jeong_23} have all showed that if $\calC$ {\em is invariant}, 
then there exists a beautiful ``transfer principle'' that characterizes $\clconv\lambda^{-1}(\calC)$ 
(i.e., the closed  convex hull of $\lambda^{-1}(\calC)$), 
namely, 
\begin{equation}
\clconv\lambda^{-1}(\calC) = \lambda^{-1}(\clconv\calC). 
\end{equation}
In addition, when $\lambda$ is the eigenvalue or singular-value map,
~\cite[Section~3.1]{Kim_22} showed that the same ``transfer principle'' also applies to $\conv\lambda^{-1}(\calC)$, 
namely 
\begin{equation}
 \conv\lambda^{-1}(\calC)= \lambda^{-1}(\conv\calC).\label{eq:transfer_conv}
\end{equation}
In fact, the ``transfer principle'' in~\eqref{eq:transfer_conv} was proved in~\cite[Theorem~3.2(d)]{Jeong_23}  for any spectral map $\lambda$, but under the additional {compactness assumption of} $\calC$. 
Lastly, it is worth mentioning that for three specific FTvN systems, a different 
characterization of $ \conv\lambda^{-1}(\calC)$ has been proposed in~\cite[Section 1]{Kim_22} --- see Remark~\ref{rmk:connection} for details. 

While the case of {invariant} $\calC$ is relatively well-studied, the results on convexifying 
$\lambda^{-1}(\calC)$ 
 for non-invariant $\calC$ are largely lacking in the literature. In this work, we aim to fill this void by developing simple and geometric characterizations of $ \conv\lambda^{-1}(\calC)$ and $ \clconv\lambda^{-1}(\calC)$ when $\calC$ has no  invariance property.  The motivation of our study is not only due to its natural theoretical interest, but also comes from two important facts.  First, non-invariant instances of $\calC$ frequently appear in the spectral constraints of many optimization problems. Indeed, one of the most common examples of $\calC$  is a general H-polyhedron $\{x\in\bbR^n:Ax\le b\}$, where 
 the $A\in\bbR^{m\times n}$ and $b\in\bbR^{m}$ has no specific structures (see e.g.,~\cite{Garner_23}). 
In such an optimization problem,  the spectral constraint  $\lambda^{-1}(\calC)$ may be non-convex, and convexifying $\lambda^{-1}(\calC)$ is 
an important step to  obtain the  convex relaxation of the entire problem.  
Second, 
as mentioned above, the definition of the invariance of $\calC$ in the context of the FTvN system involves 
another reduced system, 
which may not always exist, and hence the notion of the invariance may not always be well-defined (cf.~Remark~\ref{rmk:exist_reduced}). Therefore, it is important to develop convexification results of $\lambda^{-1}(\calC)$ without leveraging the invariance property.




%

\subsection{Main Contributions}

Our main contributions are three-fold. 

\begin{enumerate}[label=(\alph*),leftmargin=1.5\parindent]
\item We develop simple and geometric characterizations of $ \conv\lambda^{-1}(\calC)$ and $ \clconv\lambda^{-1}(\calC)$ for any feasible $\calC$ (cf.~Theorem~\ref{thm:feasible}). Our characterizations essentially only involve two geometric objects, namely $\conv (\calC\cap\calK)$ and $\calK^\circ$ (i.e., the polar cone of $\calK$). We provide examples where both geometric objects admit simple descriptions. 

\item  We identify structural assumptions on $\calC$ that allow us to strengthen the characterizations of $ \conv\lambda^{-1}(\calC)$ and $ \clconv\lambda^{-1}(\calC)$, and prove the tightness of such assumptions (cf. Theorem~\ref{thm:calC_subseteq}).
\item We specialize our results to the case of invariant $\calC$, and obtain new characterizations of $ \conv\lambda^{-1}(\calC)$ and $ \clconv\lambda^{-1}(\calC)$ in this case. 
In particular, our results improve some of the results in~\cite[Theorem~3.2]{Jeong_23} (cf.~Remark~\ref{rmk:improve}). 

\end{enumerate}

\subsection{Notations}
Let $[n]:=\{1,\ldots,n\}$. 
 For a nonempty cone $\calK\subseteq\bbR^n$, let $\calK^\circ:= \{y\in\bbW:\ipt{y}{x}\le 0,\;\forall\,x\in\calK \}$ be the polar cone of $\calK$.  
Let $\bbS_{++}^n$ 
denote the set of $n\times n$ symmetric positive definite matrices. 
Also, for $x\in\bbR^n$, define $\normt{x}_0:= \abst{\{i\in[n]:x_i\ne 0\}}$ and  $\normt{x}_p:= (\sum_{i=1}^n \abst{x_i}^p)^{1/p}$ for $p\ge 1$. 
Also, let $\bbR_+^n$ denote the nonnegative orthant, 
and define $\Rnd:=\{x\in\bbR^n:x_1\ge \ldots \ge x_n\}$ and 
\begin{equation}
\Rndp:= \bbR_+^n\cap \Rnd = \{x\in\bbR^n: x_1\ge \cdots\ge x_n \ge 0\}.  
\label{eq:Rndp}
\end{equation} 
Also, for $x\in\bbR^n$, define $\abst{x}:= (\abst{x_1},\ldots,\abst{x_n})$, and hence $|x|^\downarrow\in\Rndp$ denotes the vector obtained by rearranging entries of $|x|$ in non-increasing order. 

\noindent

 


\section{Preliminaries} \label{sec:prelim}


Let us introduce some basic results and definitions of the FTvN system in~\cite{Gowda_19,Gowda_23}, which lay the foundation of our convexification results in Section~\ref{sec:res}.

\noindent 

\begin{definition}[Orbit]
For $x\in\bbV$, define its $\lambda$-orbit by $[x]:= \{y\in\bbV:\lambda(y) =\lambda(x)\} = \lambda^{-1}(\lambda(x))$.
\end{definition}

\begin{remark} \label{rmk:compact_preimage}
Note that by the continuity and norm-preserving property of $\lambda$, if $\calC \cap \calK$ is compact (
or in particular, if $\calC$ is compact), then $\lambda^{-1}(\calC) = \lambda^{-1}(\calC\cap \calK)$ is compact. In particular, we know that $[x]= \lambda^{-1}(\lambda(x))$ is compact for all $x\in\bbV$. 
In addition, note that $\calO\subseteq\bbV$ is a $\lambda$-orbit (of some $x\in\bbV$) if and only if $\calO = \lambda^{-1}(u)$ for some $u\in\calK$. 
\end{remark}

\noindent 
Based on the notion of orbit, we introduce several equivalent definitions of a spectral set $\calE$ 
in $(\bbV,\bbW,\lambda)$ in the following lemma.
For completeness, we provide its proof 
in the Appendix~\ref{app:proof}. 

\begin{lemma}[{\cite[pp.~10]{Gowda_19}}]\label{lem:spectral_set} 
The following statements are equivalent:
\begin{enumerate}[label=(\roman*),leftmargin=1.5\parindent]
\item \label{item:spectral_i} $\calE\subseteq \bbV$ is a spectral set in $(\bbV,\bbW,\lambda)$.
\item \label{item:spectral_ii} If $x\in \calE$, then $[x]\subseteq \calE$.
\item \label{item:spectral_iii} $\calE = [\calE]:= \bigcup_{x\in\calE}\, [x]$.
\item \label{item:spectral_iv} $\calE$ is a union of $\lambda$-orbits. 
\end{enumerate}
\end{lemma}


 \noindent 
Next, we introduce several important properties of the spectral map $\lambda:\bbV\to \bbW$.

\begin{lemma}[{\cite[Section~2]{Gowda_19}}]\label{lem:gowda}
Let $( \bbV , \bbW , \lambda )$ be a FTvN system. 
Then we have the following:
\begin{enumerate}[label=(\roman*),leftmargin=1.5\parindent]
\item \label{item:ph} $\lambda ( t x ) = t \lambda ( x )$ for all $x\in\bbV$ and $t \geq 0$.
\item $\lambda:\bbV \to \bbW$ is 1-Lipschitz on $\bbV$, namely, $\| \lambda ( x ) - \lambda ( y ) \| \leq \| x - y \|$ for all $x,y\in\bbV$.
\item \label{item:sum_majorization} 
For any $c , x _ { 1 } , x _ { 2 } , \ldots , x _ { k }\in\bbV$, we have  
\begin{equation}
\left\langle \lambda (c), \lambda \left(x _ {1} + x _ {2} + \dots + x _ {k}\right) \right\rangle \leq \left\langle \lambda (c), \lambda \left(x _ {1}\right) + \lambda \left(x _ {2}\right) + \dots + \lambda \left(x _ {k}\right) \right\rangle ,
\end{equation}
or equivalently, $\lambda \left(x _ {1} + x _ {2} + \dots + x _ {k}\right) - ( \lambda \left(x _ {1}\right) + \lambda \left(x _ {2}\right) + \dots + \lambda \left(x _ {k}\right))\in\calK^\circ$. Consequently,
\begin{equation}
\normt{\lambda \left(x _ {1} + x _ {2} + \dots + x _ {k}\right)} \le  \normt{ \lambda \left(x _ {1}\right) + \lambda \left(x _ {2}\right) + \dots + \lambda \left(x _ {k}\right)}. 
\end{equation}
\item The following are equivalent for all $x,y\in\bbV$:
\begin{enumerate}
\item 
$\langle x , y \rangle = \langle \lambda ( x ) , \lambda ( y ) \rangle$,
\item $\lambda ( x + y ) = \lambda ( x ) + \lambda ( y )$,
\item $\| \lambda ( x ) - \lambda ( y ) \| = \| x - y \| .$
\end{enumerate}
\end{enumerate}
\end{lemma}




\begin{definition}[Majorization]
We say that $x$ is majorized by $y$, written as $x\prec y$, if $x\in\conv[y]$. 
\end{definition}

\begin{lemma}[{\cite[Proposition 8.3]{Gowda_23}}]\label{lem:majorization} 
For any  $x,y\in\bbV$, the following 
are equivalent:
\begin{enumerate}[label=(\roman*),leftmargin=1.5\parindent]
\item  $x\prec y$, i.e., $x\in\conv[y]$.
\item $\ipt{\lambda(c)}{\lambda(x)}\le \ipt{\lambda(c)}{\lambda(y)}$ for all $c\in\bbV$.
\item $\lambda(x) - \lambda(y)\in\calK^\circ$.
\end{enumerate}
\end{lemma}

 \noindent
The next lemma states that linear optimization over the spectral set $\lambda^{-1}(\calC)\subseteq \bbV$ can be equivalently carried out over $\calC\cap\calK\subseteq \bbW$. 

\begin{lemma}[{\cite[Corollary 3.3]{Gowda_19}}]\label{lem:sup_spectral} 
For any $c\in\bbV$ and any feasible $\calC\subseteq \bbW$, we have 
\begin{equation}
\textstyle \sup_{y\in \lambda^{-1}(\calC)}\; \ipt{c}{y} = \sup_{u\in \calC\cap\calK}\; \ipt{\lambda(c)}{u}. 
\end{equation}
\end{lemma}



\noindent
Lastly, let us introduce the notion of the {\em reduced system} of a FTvN system.

\begin{definition}[Reduced System] \label{def:reduced}
 Let $(\bbW, \bbW , \mu )$ be a FTvN system. 
We call $(\bbW, \bbW , \mu )$ a reduced system of the FTvN system $(\bbV,\bbW,\lambda)$ if  $\mu(u) = u$ for all $u\in\calK$ and $\ran\mu=\ran\lambda = \calK$. 
\end{definition}

\begin{remark} \label{rmk:mu(u)_in_[u]}
From Definition~\ref{def:reduced}, it is clear that $\mu^2 = \mu$ on $\bbW$, namely $\mu(\mu(u)) = \mu(u)$ for all $u\in\bbW$. As a result, we know that $\mu(u)\in[u]$ for all $u\in\bbW$, where $[u]:= \mu^{-1}(\mu(u))$ denotes the $\mu$-orbit of $u$. 
\end{remark}

\noindent
The reduced system provides us a  powerful and unified way to introduce the notion of invariance of a set 
$\calC\subseteq \bbW$.
\begin{definition}[Invariant Set] \label{def:invariant}
 We call $\calC\subseteq \bbW$ {\em invariant} if there exists a reduced system $(\bbW, \bbW , \mu )$ of $(\bbV,\bbW,\lambda)$ such that $\calC$ is  {\em spectral} in $(\bbW, \bbW , \mu)$, i.e.,  there exists  $\calQ\subseteq \bbW$ such that $\calC = \mu^{-1}(\calQ)$. 
\end{definition}

\begin{remark}\label{rmk:nonempty_spectral}
Note that if $\calC$ is nonempty and spectral in $(\bbW, \bbW , \mu )$, then $\calC$ is feasible. To see this, take any $u\in \calC$, and  from Remark~\ref{rmk:mu(u)_in_[u]} and Lemma~\ref{lem:spectral_set}\ref{item:spectral_ii}, we know that $\mu(u)\in[u]\subseteq \calC$. Since $\mu(u)\in\calK$, we know that  $\mu(u)\in \calC\cap \calK$, and hence $\calC\cap \calK\ne \emptyset$. 
\end{remark}

\begin{example}\label{eg:inv}
To illustrate Definition~\ref{def:invariant}, let us consider the following examples. 
\begin{enumerate}[label=(\roman*),leftmargin=1.5\parindent]
\item $(\bbV,\bbW,\lambda)$ is the FTvN system in~\ref{eg:S^n} 
or $(\bbR^n,\bbR^n,\lambda)$ with $\lambda(x) = x^\downarrow$: the reduced system $(\bbW, \bbW , \mu )=(\bbR^n,\bbR^n,\mu)$ with $\mu(u) = u^\downarrow$ for $u\in \bbR^n$. In this case, $\calC\subseteq \bbW$ is spectral in  $(\bbW, \bbW , \mu )$ if and only if it is {\em permutation-invariant}. 
\item $(\bbV,\bbW,\lambda)$ is the FTvN system in~\ref{eg:R^m_n} 
or $(\bbR^d,\bbR^d,\lambda)$ with $\lambda(x) = |x|^\downarrow$: the reduced system $(\bbW, \bbW , \mu )=(\bbR^d,\bbR^d,\mu)$ with $\mu(u) = |u|^\downarrow$  for $u\in \bbR^d$. In this case, $\calC\subseteq \bbW$ is spectral in  $(\bbW, \bbW , \mu )$ if and only if it is {\em permutation- and sign-invariant}. 
\item $(\bbV,\bbW,\lambda)=(\bbR^d,\bbR^d,\lambda)$ with $\lambda(x) = |x|$: the reduced system $(\bbW, \bbW , \mu )$ is itself.  In this case, $\calC\subseteq \bbW$ is spectral in  $(\bbW, \bbW , \mu )$ if and only if it is {\em sign-invariant}. 
\end{enumerate}
\end{example}

\begin{remark}\label{rmk:exist_reduced}
As noted in~\cite[Section~10]{Gowda_23}, given any FTvN system $(\bbV,\bbW,\lambda)$, its  reduced system is not guaranteed to exist, 
and a counterexample would be the FTvN system induced by a general complete and isometric  hyperbolic polynomial.  This reveals a fundamental limitation of the notion of invariance --- it may not even be  well-defined! Therefore, as far as convexification is concerned, it is important to develop results without leveraging the notion of invariance. 
\end{remark}


\section{Main Results} \label{sec:res}



Let us start with the following important result, which is the cornerstone in establishing other results in this work.  

\begin{theorem} \label{thm:main}
Let $\mathcal { D }  \subseteq \calK$ be nonempty and convex.  Then $\lambda^{-1}(\calD + \mathcal {K} ^ {\circ})$ is convex, and moreover, 
\begin{equation}
\mathsf{c o n v} \lambda^ {- 1} (\mathcal {D}) = \lambda^{-1}(\calD+\calK^\circ) = \{x \in \mathbb {V}: \exists\; u \in \mathcal {D}\st \lambda (x) - u \in \mathcal {K} ^ {\circ} \}. \label{eq:conv_D}
\end{equation} 
As a result, $\mathsf{clc o n v} \lambda^ {- 1} (\mathcal {D}) = \cl\lambda^{-1}(\calD+\calK^\circ)$, and if 
$\calD+\calK^\circ$ is closed, then 
\begin{equation}
\mathsf{clc o n v} \lambda^ {- 1} (\mathcal {D}) = \lambda^{-1}(\calD+\calK^\circ). \label{eq:conv_eq_lambda-1}
\end{equation}
\end{theorem}

\noindent
The proof of Theorem~\ref{thm:main} leverages the following simple lemma. 

 \begin{lemma} \label{lem:conv_preimage}
For any $u \in { \cal K }$, we have $\conv\lambda^{-1}( u ) = \lambda^{-1} ( u + \calK^{\circ}  )$. 
\end{lemma}

\begin{proof}
 Let $y \in \lambda^{-1} (u)$. Note that $x \in \lambda^{-1}(  u + \calK ^ { \circ } )$ if and only if $\lambda ( x ) - \lambda(y)=\lambda ( x ) - u \in \calK ^ { \circ }$, which amounts to $x\prec y$ 
  by Lemma~\ref{lem:majorization}. By definition, this amounts to  $x\in \conv[y] = \conv\lambda^{-1}(u)$. 
\end{proof}

\begin{proof}[Proof of Theorem~\ref{thm:main}]
We first show that $\lambda^{-1}(\calD + \mathcal {K} ^ {\circ})$ is convex.  Fix any $x _ { 1 } , x _ { 2 } \in \lambda^{-1}(\calD + \mathcal {K} ^ {\circ})$  and any $t \in [ 0 , 1 ]$, so  that there exist $u _ { 1 } , u _ { 2 } \in \mathcal { D }$ such that $\lambda(x _ { i }) - u_i \in \mathcal {K} ^ {\circ}$ for $i=1,2$. 
Since $\mathcal {K} ^ {\circ}$ is convex,  we have 
\begin{equation}
(t\lambda(x _ {1 }) + (1-t) \lambda(x _ { 2 })) - (tu_1 + (1-t)u_2 ) =t(\lambda(x _ {1 }) - u_1) + (1-t) (\lambda(x _ { 2 }) - u_2)\in \mathcal {K} ^ {\circ}.  \label{eq:cvx_comb}
\end{equation}
 Also,  by Lemma~\ref{lem:gowda}\ref{item:ph} and~\ref{item:sum_majorization}, we have 
\begin{equation}
 \lambda ( t x _ { 1 } + (1-t) x _ { 2 } ) - ( t \lambda ( x _ { 1 } ) + (1-t) \lambda ( x _ { 2 } ))  \in \calK^\circ. \label{eq:major_conv_comb}
\end{equation}
Since $\mathcal {K} ^ {\circ}$ is a convex cone and $\calD$ is convex, combining~\eqref{eq:cvx_comb} and~\eqref{eq:major_conv_comb}, we have  
\begin{equation}
\lambda ( t x _ { 1 } + (1-t) x _ { 2 } ) \in (tu_1 + (1-t)u_2 ) + \mathcal {K} ^ {\circ} \subseteq \calD + \mathcal {K} ^ {\circ},
\end{equation}
or equivalently, $t x _ { 1 } + (1-t) x _ { 2 }\in \lambda^{-1}(\calD + \mathcal {K} ^ {\circ})$. 

Next, note that  the second equality in~\eqref{eq:conv_D} follows directly from the definition, i.e., 
\begin{align}
\{x \in \mathbb {V}: \exists\; u \in \mathcal {D}&\st \lambda (x) - u \in \mathcal {K} ^ {\circ} \} = \{x \in \mathbb {V}: \exists\; u \in \mathcal {D}\st  x  \in \lambda^{-1}(u+\mathcal {K} ^ {\circ}) \}\\
& ={\bigcup}_ {u \in \mathcal {D}}\; \lambda^{-1}(u+\mathcal {K} ^ {\circ}) =\lambda^{-1}\Big({\bigcup}_ {u \in \mathcal {D}}\;  (u+\mathcal {K} ^ {\circ})\Big) = \lambda^{-1}(\mathcal {D}+\mathcal {K} ^ {\circ}). \label{eq:lambda-1_D+Kcirc}
\end{align}
To show the first equality in~\eqref{eq:conv_D}, first note that since $\lambda^ {- 1} (\mathcal {D}) \subseteq  \lambda^{-1}(\calD+\calK^\circ)$ and $\lambda^{-1}(\calD+\calK^\circ)$ is convex, we have $\conv\lambda^ {- 1} (\mathcal {D}) \subseteq  \lambda^{-1}(\calD+\calK^\circ)$. Next, note that by Lemma~\ref{lem:conv_preimage},   for any $u\in\calD$, we have $\lambda^{-1}(u+\mathcal {K} ^ {\circ}) = \conv\lambda^{-1}( u )\subseteq\conv\lambda^{-1}( \calD )$. Hence, by~\eqref{eq:lambda-1_D+Kcirc}, we have $\lambda^{-1}(\mathcal {D}+\mathcal {K} ^ {\circ})\subseteq \conv\lambda^{-1}( \calD )$. 
Lastly, note that if $\calD+\calK^\circ$ is closed, by the continuity of $\lambda$, we know that $\lambda^{-1}(\calD+\calK^\circ)$ is closed, and we arrive at~\eqref{eq:conv_eq_lambda-1}. 
\end{proof}

\begin{remark}[Closedness of $\calD+\calK^\circ$]
Note that two common conditions that ensure $\calD+\calK^\circ$ to be closed are i) $\calD$ is compact and ii) both $\calD$ and $\calK$ are polyhedral. 
\end{remark}

\subsection{First Result}

\noindent 
As mentioned in the introduction, our interest is certainly not limited to convexifying $\lambda^ {- 1} (\mathcal {D})$ for nonempty and convex $\calD\subseteq \calK$, but actually for any feasible $\calC\subseteq \bbW$. This leads to the following theorem. 

\begin{theorem} \label{thm:feasible}
Let $\calC\subseteq \bbW$ be feasible. Then 
\begin{equation}
\conv\lambda^{-1}(\calC) 
= \lambda^{-1}(\conv (\calC\cap\calK)+\calK^\circ) = \{x \in \mathbb {V}: \exists\; u \in\conv (\calC\cap\calK)\st \lambda (x) - u \in \mathcal {K} ^ {\circ} \}.  \label{eq:feasible_conv_hull}
\end{equation}
Also, for any convex set $\calD$ such that $\conv (\calC\cap\calK) \subseteq \calD\subseteq \clconv (\calC\cap\calK)$, we have
\begin{equation}
\clconv\lambda^{-1}(\calC) 
= \cl\lambda^{-1}(\calD+\calK^\circ). \label{eq:closed_cvx_hull} 
\end{equation}
In particular, if $\calD+\calK^\circ$ is closed, then $\clconv\lambda^{-1}(\calC) = \lambda^{-1}(\calD+\calK^\circ)$. 
\end{theorem}

\noindent
Before proving Theorem~\ref{thm:feasible}, let us provide some  examples and remarks. 

\begin{example} \label{eg:two_pt}
To illustrate Theorem~\ref{thm:feasible}, let us consider 
 the FTvN system $(\bbR^2,\bbR^2,\lambda)$, where $\lambda(x) = x^\downarrow$. In this case, $\calK = \bbR^2_\downarrow$ and $\calK^\circ = \{u\in\bbR^2: u_1\le 0,\; u_2= -u_1\} = \{ (-\alpha,\alpha):\alpha\ge 0\}$. Let $\calC:= \{(1,0),(1,2)\}$. Then 
\begin{equation}
\conv\lambda^{-1}(\calC) = \conv\lambda^{-1}(\calC\cap\calK) = \conv\lambda^{-1}((1,0)) = \conv\{(1,0),(0,1)\} = \Delta_2. 
\nn
\end{equation}
Since $\conv\lambda^{-1}(\calC)$ is a closed line-segment, we clearly have $\clconv\lambda^{-1}(\calC) = \conv\lambda^{-1}(\calC)$. 
In addition, 
\begin{align*}
\lambda^{-1}(\conv (\calC\cap\calK)+\calK^\circ) &= \lambda^{-1}(\{ (1-\alpha,\alpha):\alpha\ge 0\})\\
 &=  \lambda^{-1}(\{ (1-\alpha,\alpha):0\le \alpha\le 1/2\}) =\{ (\alpha,1-\alpha):0\le \alpha\le 1\} = \Delta_2.  
\end{align*}
Therefore~\eqref{eq:feasible_conv_hull} holds. 
Note that although the setting above is extremely simple, 
examples of this type turn out to be quite effective in offering geometric intuitions and constructing counterexamples of our various convexifications of $\lambda^{-1}(\calC)$ in 
the following. 
\end{example}

\begin{remark} \label{rmk:cl_nec}
Note that the closure operation on the right-hand side of~\eqref{eq:closed_cvx_hull} is necessary in general,
since there exist examples where $\lambda^{-1}(\calD+\calK^\circ)$ is not closed if  $\calD+\calK^\circ$ is not closed. 
To see this, consider again the FTvN system in Example~\ref{eg:two_pt}, but let $\calC = \{(u_1,u_2): u_2^2 \le u_1^2-1, \;u_1\ge 1\}\subseteq\calK:= \bbR^2_\downarrow$, which is closed and convex. In this case, $\conv (\calC\cap\calK) =\clconv (\calC\cap\calK) = \calC$, and hence $\calD = \calC$. However, note that $\calD+\calK^\circ = \calH:=\{(u_1,u_2): u_2 > -u_1\}$, which is an open half-space (and hence  not closed). Moreover, note that $\lambda^{-1}(\calD+\calK^\circ) = \calH$, which is of course, not closed. 
\end{remark}

\begin{remark} \label{rmk:conv_before_after}
Perhaps as a   subtle point, note that we cannot replace $\conv (\calC\cap\calK)$ by $(\conv \calC)\cap\calK$ in~\eqref{eq:feasible_conv_hull} in general. For example, in Example~\ref{eg:two_pt}, 
we have $(\conv \calC)\cap\calK = \{(1,\beta):0\le \beta\le 1\}$. Thus 
\begin{align*}
\lambda^{-1}((\conv \calC)\cap\calK+\calK^\circ) &= \lambda^{-1}(\{ (1-\alpha,\beta+\alpha):\alpha\ge 0, \;0\le \beta\le 1\})\\
 &=  \lambda^{-1}(\{ (u_1,u_2):u_2\le u_1\le 1,\; u_1+u_2\ge 1\})\\
 & =\{ (u_1,u_2):u_1, u_2\le 1,\; u_1+u_2\ge 1\} . 
\end{align*}
As we can see, $\conv\lambda^{-1}(\calC) \ne  \lambda^{-1}(((\conv \calC)\cap\calK)+\calK^\circ)$ in this case. 
On the other hand, note that if $\calC:= \{(1,0),(0,1)\}$, then $\conv\lambda^{-1}(\calC) = \Delta_2$ and 
\begin{align*}
\lambda^{-1}(((\conv \calC)\cap\calK)+\calK^\circ) &= \lambda^{-1}(\{ (\beta-\alpha,1-\beta+\alpha):\alpha\ge 0, \;1/2\le \beta\le 1\})\\
 &=  \lambda^{-1}(\{ (u,1-u):1/2\le u\le 1\}) = \Delta_2 = \conv\lambda^{-1}(\calC). 
\end{align*}
This reveals that under certain conditions, we can indeed  replace $\conv (\calC\cap\calK)$ by $(\conv \calC)\cap\calK$ in~\eqref{eq:feasible_conv_hull}. Such conditions will be revealed in Section~\ref{sec:second}. 
\end{remark}

\begin{example} \label{eg:singular_value}
From~\eqref{eq:feasible_conv_hull}, we see that if both $\conv (\calC\cap\calK)$ and $\calK^\circ$ admit simple descriptions, then $\conv\lambda^{-1}(\calC)$ can be simply described. 
As an example, consider
~\ref{eg:R^m_n}, where $\calK = \Rndp$ 
and 
\begin{equation}
\calK^\circ =  \{y\in\bbR^n: \textstyle\sum_{i=1}^l y_i \le 0, \; \forall\,l\in[n]\}. 
\end{equation}
Let $A\in\bbS_{++}^n$ and $\calC$ be the $k$-sparse ellipsoid induced by $A$, i.e.,  
\begin{equation}
\calC:= \{u\in\bbR^n:\normt{u}_0\le k,\;\;  {u}^\top A u\le 1\},
\end{equation}
where $k\in[n]$. Note that $\calC\cap\calK$ has the following concise representation:
\begin{equation}
\calC\cap\calK = \{u\in\bbR^n:u_1\ge \cdots\ge  u_k\ge 0,\;\; u_{k+1}=\cdots=u_n = 0,\;\;  {u}^\top A u\le 1\},
\end{equation}
and it is clear that $\calC\cap\calK$ is nonempty, convex and compact. As such, from Theorem~\ref{thm:feasible}, we have 
\begin{align}
\conv\lambda^{-1}(\calC) &= \clconv\lambda^{-1}(\calC) = \lambda^{-1}((\calC\cap\calK)+\calK^\circ) = \{X \in \bbR^{m\times n}: \;\exists\, v\in \bbR^k \st (X,v)\in \calP\},\nn
\end{align}
{where}
\begin{align}
\begin{split}
\calP:= \{(X,v) \in \bbR^{m\times n}\times \bbR^k:\;\; & v_1\ge \cdots \ge v_k\ge 0,\;\;  {v}^\top [A]_{k,k}\, v\le 1,\\
& \qquad \sigma_l(X) - \textstyle \sum_{i=1}^{\min\{l,k\}} v_i\le 0,\;\; \forall\,l\in[n] \}.
\end{split} \label{eq:P}
\end{align}
Here $ [A]_{k,k}\in\bbS_{++}^k$ denotes the $k\times k$ leading principle sub-matrix of $A$, and $\sigma_l(X)$ denotes the Ky-Fan $l$-norm of $X$, i.e., the sum of the $l$ largest singular values of $X$. Note that $\calP$ is nonempty, convex and compact, and $\conv\lambda^{-1}(\calC) = \clconv\lambda^{-1}(\calC)$ is the projection of $\calP$ onto the $X$-coordinate, which is clearly  nonempty, convex and compact. 
\end{example}

\begin{remark}
Let us remark how Theorem~\ref{thm:feasible} plays a role in convexifying optimization problems involving ``spectral constraints''. For simplicity,  let us consider the following optimization problem:
\begin{equation}
\min\; f(x)\quad \st x\in\calS, \quad \lambda(x)\in\calC, \tag{CP} \label{eq:cvx_Prob}
\end{equation}
where $f:\bbV\to \bbR\cup\{+\infty\}$ is a proper, closed and convex function, $\calS\subseteq \bbV$ is  nonempty, closed and convex, and $\calC\subseteq \bbW$ is feasible. To convexify the ``spectral constraint'' $\lambda(x)\in\calC$, we replace it with $x\in \conv\lambda^{-1}(\calC) = \lambda^{-1}(\calD+\calK^\circ)$. From the second equality in~\eqref{eq:feasible_conv_hull}, the ``convexified''~\eqref{eq:cvx_Prob} reads: 
\begin{equation}
\min\; f(x)\quad \st x\in\calS, \quad u\in \conv(\calC\cap\calK),\quad \lambda (x) - u \in \mathcal {K} ^ {\circ}, \label{eq:cvx_Prob2}
\end{equation}
which is an optimization problem over $(x,u)\in\bbV\times\bbW$. In many situations, the constraint $\lambda (x) - u \in \mathcal {K} ^ {\circ}$ encodes a nonempty, closed and convex set in $(x,u)$, and furthermore, if $\calC\cap\calK$ is compact, then $\conv(\calC\cap\calK)$ is nonempty, convex and compact. Therefore, the ``convexified'' problem in~\eqref{eq:cvx_Prob2} becomes a convex optimization problem in ``standard'' format, namely minimizing a proper, closed and convex function over a closed and convex set.
As a concrete example, consider the setting in Example~\ref{eg:singular_value}, and then~\eqref{eq:cvx_Prob2} becomes
\begin{align*}
\min\; f(X)\;\; \st X\in\calS, \;\;  v_1\ge \cdots \ge v_k\ge 0,\;\;  {v}^\top [A]_{k,k}\, v\le 1,
 \;\; \sigma_l(X) - \textstyle \sum_{i=1}^{\min\{l,k\}} v_i\le 0,\;\; \forall\,l\in[n]. 
\end{align*}
Indeed, if $\calS = \bbR^{m\times n}$, then the constraint set of this problem simply involves $k$ linear inequalities, one quadratic inequality and $n$ inequalities induced by  closed convex functions. 
\end{remark}

\begin{remark}\label{rmk:connection} 
Note that under the proper notion of the invariance of $\calC$ (cf.~Example~\ref{eg:inv}), the result in  
\eqref{eq:feasible_conv_hull} has been established in~\cite[Section~1]{Kim_22} for three specific FTvN systems $(\bbR^n,\bbR^n,\lambda_i)$, $i=1,2,3$, where for $x\in\bbR^n$, 
\begin{equation}
\lambda_1 (x):= x^\downarrow, \quad \lambda_2 (x):= |x|\quad\andd\quad  \lambda_3 (x):= |x|^\downarrow. \label{eq:three_FTvN} 
\end{equation}
(In fact,  
the geometric condition $\lambda (x) - u \in \mathcal {K} ^ {\circ}$ 
was 
stated algebraically 
in~\cite{Kim_22}.) 
Note that in~\cite{Kim_22}, the results for these three cases were proved in a case-by-case manner by critically exploiting the invariance property of $\calC$ and the  (weak) majorization theory~\cite{Marshall_11}. In contrast, by leveraging the framework of the FTvN system, our result in~\eqref{eq:feasible_conv_hull} significantly improves the results in~\cite{Kim_22}, in two senses: 
i) the same convexification result holds even if $\calC$ {\em has no invariance property at all}, and 
ii) the result holds for {\em any} FTvN system (and not limited to the three FTvN systems in~\eqref{eq:three_FTvN}). 
\end{remark}


\noindent
Let us now prove Theorem~\ref{thm:feasible}. To that end, we need the following lemma. 

\begin{lemma} \label{lem:lambda-1_conv}
Let $\calC\subseteq \bbW$ be feasible. Then $\lambda^{-1}(\conv (\calC\cap\calK))\subseteq \conv\lambda^{-1}(\calC)$. 
\end{lemma}

\begin{proof}
Fix any $x\in \lambda^{-1}(\conv (\calC\cap\calK))$. Then there exist $\{u_i\}_{i=1}^k\subseteq \calC\cap\calK$ and $(t_1,\ldots,t_k)\in\Delta_k$ such that $\lambda(x) = \sum_{i=1}^k t_iu_i$. In addition, by~\ref{P2}, for any $i\in[k]$, there exists $x_i\in\lambda^{-1}(u_i)$ such that $\ipt{x}{x_i} = \ipt{\lambda(x)}{\lambda(x_i)}=
\ipt{\lambda(x)}{u_i}$. As a result, by the norm-preserving property of $\lambda$, we have 
\begin{equation}
\textstyle \normt{x}\normt{\sum_{i=1}^k t_ix_i}\ge \ipt{x}{\sum_{i=1}^k t_ix_i} = \sum_{i=1}^k t_i \ipt{x}{x_i} = \sum_{i=1}^k t_i \ipt{\lambda(x)}{u_i} = \normt{\lambda(x)}^2 = \normt{x}^2, \label{eq:CS} 
\end{equation}
which implies that $\normt{\sum_{i=1}^k t_ix_i}\ge \normt{x}$. On the other hand, by Lemma~\ref{lem:gowda}\ref{item:sum_majorization}, we have 
\begin{equation}
\textstyle\normt{\sum_{i=1}^k t_ix_i} = \normt{\lambda(\sum_{i=1}^k t_ix_i)}\le \normt{\sum_{i=1}^k t_i\lambda(x_i)} = \normt{\sum_{i=1}^k t_iu_i}= \normt{\lambda(x)} = \normt{x}. 
\end{equation}
Thus $\normt{\sum_{i=1}^k t_ix_i}= \normt{x}$ and by~\eqref{eq:CS}, we have  $\normt{x}\normt{\sum_{i=1}^k t_ix_i}= \ipt{x}{\sum_{i=1}^k t_ix_i}$. This implies that ${x}={\sum_{i=1}^k t_ix_i}$. Since $x_i\in\lambda^{-1}(u_i)\subseteq \lambda^{-1}(\calC\cap\calK)=\lambda^{-1}(\calC)$ for $i\in[k]$, we have $x\in \conv \lambda^{-1}(\calC)$, and consequently, 
$\lambda^{-1}(\conv (\calC\cap\calK))\subseteq\conv \lambda^{-1}(\calC).$
\end{proof}

\noindent
To prove Theorem~\ref{thm:feasible}, we also need the following proposition, which equates $\conv\lambda^{-1}(\calC)$ (resp.\ $\clconv\lambda^{-1}(\calC)$) with $\conv\lambda^{-1}(\calD)$ (resp.\ $\clconv\lambda^{-1}(\calD)$) for some nonempty $\calD\subseteq \calK$. 


\begin{prop}\label{prop:comv_hull_eq}
Let $\calC\subseteq \bbW$ be feasible. Then 
\begin{enumerate}[label=(\roman*),leftmargin=1.5\parindent]
\item \label{item:conv_eq_i} for any $\calD$ that satisfies  $\calC\cap\calK\subseteq \calD \subseteq \conv (\calC\cap\calK)$, we have $\conv\lambda^{-1}(\calC) = \conv\lambda^{-1}(\calD)$, and  
\item \label{item:conv_eq_ii} for any $\calD$ that satisfies $\calC\cap\calK\subseteq \calD \subseteq \clconv (\calC\cap\calK)$, we have $\clconv\lambda^{-1}(\calC) = \clconv\lambda^{-1}(\calD)$. 
\end{enumerate}
\end{prop}

\begin{proof}
First, if $\calC\cap\calK\subseteq \calD$, then $\lambda^{-1}(\calC) = \lambda^{-1}(\calC\cap\calK)\subseteq \lambda^{-1}(\calD)$, and $\conv\lambda^{-1}(\calC) \subseteq \conv\lambda^{-1}(\calD)$.
On the other hand, by Lemma~\ref{lem:lambda-1_conv}, we have 
 $\lambda^{-1}(\conv (\calC\cap\calK))\subseteq \conv\lambda^{-1}(\calC)$. 
Consequently, for all $\calD \subseteq \conv (\calC\cap\calK)$, we have $\conv\lambda^{-1}(\calD)\subseteq \conv\lambda^{-1}(\calC)$, which  finishes the proof of Part~\ref{item:conv_eq_i}.    

Next, we prove Part~\ref{item:conv_eq_ii}. Using the same reasoning as above, if $\calC\cap\calK\subseteq \calD$, then $\clconv\lambda^{-1}(\calC) \subseteq \clconv\lambda^{-1}(\calD)$. Next, we show that $\lambda^{-1}(\clconv (\calC\cap\calK))\subseteq \clconv\lambda^{-1}(\calC)$,  which then implies that 
for all $\calD \subseteq \clconv (\calC\cap\calK)$, we have $\clconv\lambda^{-1}(\calD)\subseteq \clconv\lambda^{-1}(\calC)$, thereby  finishing the proof. Take any $x\in\lambda^{-1}(\clconv (\calC\cap\calK)), $ and suppose that $x\not\in \clconv\lambda^{-1}(\calC)$. Then by strong separation, there exists $c\in\bbV$ such that 
\begin{align}
\textstyle \ipt{c}{x} > \sup_{y\in \clconv\lambda^{-1}(\calC)}\; \ipt{c}{y} = \sup_{y\in \lambda^{-1}(\calC)}\; \ipt{c}{y} = \sup_{u\in \calC\cap\calK}\; \ipt{\lambda(c)}{u}. \label{eq:cx>}
\end{align}
On the other hand, since $\lambda(x)\in \clconv (\calC\cap\calK)$, we have 
\begin{align}
\textstyle \ipt{c}{x}\le \ipt{\lambda(c)}{\lambda(x)}\le \sup_{u\in \clconv (\calC\cap\calK)}\; \ipt{\lambda(c)}{u} = \sup_{u\in \calC\cap\calK}\; \ipt{\lambda(c)}{u}.\label{eq:cx<=}
\end{align}
Combining~\eqref{eq:cx>} and~\eqref{eq:cx<=}, we reach a contradiction.  
\end{proof}

\begin{proof}[Proof of Theorem~\ref{thm:feasible}]
Straightforward combination of Theorem~\ref{thm:main} and Proposition~\ref{prop:comv_hull_eq}. 
\end{proof}

\subsection{Second Result}\label{sec:second}

As we have seen from Remark~\ref{rmk:conv_before_after}, although in general we  cannot replace $\conv (\calC\cap\calK)$ by $(\conv \calC)\cap\calK$ in the characterization of $\conv\lambda^{-1}(\calC)$ in~\eqref{eq:feasible_conv_hull}, we can indeed do so under certain conditions. In fact, the same also applies to  the characterization of $\clconv\lambda^{-1}(\calC)$  in~\eqref{eq:closed_cvx_hull}.  In this section, we will identify such conditions, and also prove their tightness.  

\begin{theorem}\label{thm:calC_subseteq}
Let $\calC\subseteq\bbW$ be feasible.  
\begin{enumerate}[label=(\roman*),leftmargin=1.5\parindent]
\item \label{item:conv_eq2_i} The following three statements are equivalent: 
\begin{enumerate}[label=(\alph*),leftmargin=1\parindent]
\item \label{item:conv_eq2_i_a} $(\conv\calC)\cap\calK\subseteq \conv(\calC\cap\calK) + \calK^\circ$.
\item \label{item:conv_eq2_i_b} For any convex $\calD$ that satisfies $\conv(\calC\cap\calK)\subseteq \calD \subseteq (\conv \calC)\cap\calK$, we have $\conv\lambda^{-1}(\calC) = \lambda^{-1}(\calD+\calK^\circ).$
\item \label{item:conv_eq2_i_c} 
$\conv\lambda^{-1}(\calC) = \lambda^{-1}(((\conv \calC)\cap\calK)+\calK^\circ).$
\end{enumerate}
\item \label{item:conv_eq2_ii} 
The following three statements are equivalent: 
\begin{enumerate}[label=(\alph*),leftmargin=1\parindent]
\item \label{item:conv_eq2_ii_a} $(\clconv\calC)\cap\calK\subseteq \cl(\clconv(\calC\cap\calK) + \calK^\circ)$.
\item \label{item:conv_eq2_ii_b} For any convex $\calD$ that satisfies $\conv(\calC\cap\calK)\subseteq \calD \subseteq (\clconv \calC)\cap\calK$, we have $\clconv\lambda^{-1}(\calC) = \cl\lambda^{-1}(\calD+\calK^\circ).$
\item \label{item:conv_eq2_ii_c} 
$\clconv\lambda^{-1}(\calC) = \cl\lambda^{-1}(((\clconv \calC)\cap\calK)+\calK^\circ).$
\end{enumerate}
\end{enumerate}
\end{theorem}


\begin{remark}
Note that Theorem~\ref{thm:calC_subseteq}\ref{item:conv_eq2_i}\ref{item:conv_eq2_i_b} is a strengthened version of the characterization of $\conv\lambda^{-1}(\calC)$ 
 in~\eqref{eq:feasible_conv_hull}, 
 in the sense that $\calD$ can be as large as $(\conv \calC)\cap\calK$. This, of course, requires stronger assumption on $\calC$, which is shown  in Theorem~\ref{thm:calC_subseteq}\ref{item:conv_eq2_i}\ref{item:conv_eq2_i_a}. In addition, the equivalence between~\ref{item:conv_eq2_i_a} and~\ref{item:conv_eq2_i_b} in Theorem~\ref{thm:calC_subseteq}\ref{item:conv_eq2_i} indicates that this assumption is tight.  
 Note that the same also applies to the strengthened characterization of $\clconv\lambda^{-1}(\calC)$ in Theorem~\ref{thm:calC_subseteq}\ref{item:conv_eq2_ii}\ref{item:conv_eq2_ii_b} and its required assumption in  Theorem~\ref{thm:calC_subseteq}\ref{item:conv_eq2_ii}\ref{item:conv_eq2_ii_a}. 
\end{remark}

\noindent 
The proof of Theorem~\ref{thm:calC_subseteq} leverages the following results. 

 \begin{prop}\label{prop:conv_hull_eq2}
Let $\calC\subseteq \bbW$ be feasible.  
\begin{enumerate}[label=(\roman*),leftmargin=1.5\parindent]
\item \label{item:conv_eq3_i} If $(\conv\calC)\cap\calK\subseteq \conv(\calC\cap\calK) + \calK^\circ$, then 
for any $\calD$ that satisfies $\calC\cap\calK\subseteq \calD \subseteq \conv \calC$, we have 
$$\conv\lambda^{-1}(\calC) = \conv\lambda^{-1}(\calD).$$ 
\item \label{item:conv_eq3_ii} If $(\clconv\calC)\cap\calK\subseteq \cl(\clconv(\calC\cap\calK) + \calK^\circ)$, then for any $\calD$ that satisfies $\calC\cap\calK\subseteq \calD \subseteq \clconv \calC$, we have $$\clconv\lambda^{-1}(\calC) = \clconv\lambda^{-1}(\calD).$$
\end{enumerate}
\end{prop}
 
\begin{proof}
Using the same reasoning as in the proof of Proposition~\ref{prop:comv_hull_eq}, if $\calC\cap\calK\subseteq \calD$, then $\conv\lambda^{-1}(\calC) \subseteq \conv\lambda^{-1}(\calD)$ and $\clconv\lambda^{-1}(\calC) \subseteq \clconv\lambda^{-1}(\calD)$. Therefore, to show Part~\ref{item:conv_eq2_i}, it suffices to show $\lambda^{-1}(\conv \calC) \subseteq  \conv\lambda^{-1}(\calC)$, which then implies that $\conv\lambda^{-1}(\calD) \subseteq  \conv\lambda^{-1}(\calC)$ for any $\calD \subseteq \conv \calC$. To that end, take any $x\in \lambda^{-1}(\conv \calC)=\lambda^{-1}((\conv \calC)\cap\calK)$, and by assumption, we have
$$ x\in \lambda^{-1}((\conv \calC)\cap\calK)\subseteq \lambda^{-1}(\conv(\calC\cap\calK) + \calK^\circ).$$ 
Then there exist $v\in \conv(\calC\cap\calK)\subseteq \calK$ such that $\lambda(x) - v\in\calK^\circ$. Let $y\in\lambda^{-1}(v)$, then by Lemma~\ref{lem:majorization}, we know that $x\prec y$, 
or equivalently, $x\in\conv[y]$. Since $[y] = \lambda^{-1}(v)\subseteq \lambda^{-1}(\conv(\calC\cap\calK))\subseteq \conv\lambda^{-1}(\calC)$ (cf.~Lemma~\ref{lem:lambda-1_conv}), we have $\conv [y]\subseteq \conv\lambda^{-1}(\calC)$, and hence $x\in \conv\lambda^{-1}(\calC)$.


To show Part~\ref{item:conv_eq2_ii}, it suffices to show 
$\lambda^{-1}(\clconv \calC)\subseteq \clconv\lambda^{-1}(\calC)$, which then implies that $\clconv\lambda^{-1}(\calD) \subseteq  \clconv\lambda^{-1}(\calC)$ for any $\calD \subseteq \clconv \calC$. Let $x\in \lambda^{-1}(\clconv \calC)$. If $x\not\in \clconv\lambda^{-1}(\calC)$, then by strong separation, there exists $c\in\bbV$ such that~\eqref{eq:cx>} holds.   
On the other hand, since $\lambda(x)\in (\clconv \calC)\cap\calK\subseteq \cl(\clconv(\calC\cap\calK) + \calK^\circ)$, we have 
\begin{align}
\begin{split}
 &\textstyle\ipt{c}{x}\le \ipt{\lambda(c)}{\lambda(x)}\le \sup_{u\in \cl(\clconv(\calC\cap\calK) + \calK^\circ)}\; \ipt{\lambda(c)}{u}\\ 
 &\textstyle\qquad\qquad = \sup_{v\in \clconv(\calC\cap\calK)}\; \ipt{\lambda(c)}{v}+\sup_{w\in \calK^\circ}\;\ipt{\lambda(c)}{w}= \sup_{v\in \calC\cap\calK}\; \ipt{\lambda(c)}{v},
 \end{split} \label{eq:ub_ipt_cx}
\end{align}
where the 
last inequality in~\eqref{eq:ub_ipt_cx} follows from  
$\lambda(c)\in\calK$. 
Combining~\eqref{eq:cx>} and~\eqref{eq:ub_ipt_cx}, we reach a contradiction. 
\end{proof}

\begin{lemma} \label{lem:A_B}
For any nonempty sets $\calA, \calB\subseteq\calK$, if $\lambda^{-1}(\calA+\calK^\circ) = \lambda^{-1}(\calB+\calK^\circ)$, then $\calB \subseteq\calA + \calK^\circ$ and $\calA \subseteq\calB + \calK^\circ$. 
\end{lemma}

\begin{proof}
By the symmetry between $\calA$ and $\calB$, we only need to show $\calB \subseteq\calA + \calK^\circ$. 
Since $\lambda^{-1}(\calA+\calK^\circ) = \lambda^{-1}(\calB+\calK^\circ)$, we have 
\begin{equation}
(\calA+\calK^\circ)\cap\calK = \lambda(\lambda^{-1}(\calA+\calK^\circ)) = \lambda(\lambda^{-1}(\calB+\calK^\circ)) = (\calB+\calK^\circ)\cap\calK. 
\end{equation}
Suppose that $\calB \not\subseteq\calA + \calK^\circ$, then there exists $u\in\calB$ such that $u\not\in\calA + \calK^\circ$, and hence $u\not\in(\calA + \calK^\circ)\cap\calK$. On the other hand, since $\calB\subseteq \calK$ and $0\in\calK^\circ$, 
we have  $u\in\calB \subseteq  (\calB+\calK^\circ)\cap\calK=(\calA+\calK^\circ)\cap\calK$, leading to a contradiction. 
\end{proof}

\begin{lemma} \label{lem:B}
For any nonempty set $\calB\subseteq\calK$, we have $((\calB + \calK^\circ)\cap \calK) + \calK^\circ = \calB + \calK^\circ$. 
\end{lemma}

\begin{proof}
Since $0\in \calK^\circ$ and $\calB\subseteq\calK$, we have $\calB 
\subseteq (\calB +\calK^\circ)\cap \calK$, and hence $\calB + \calK^\circ \subseteq ((\calB +\calK^\circ)\cap \calK)+ \calK^\circ$. Conversely, since $(\calB + \calK^\circ)\cap \calK \subseteq\calB + \calK^\circ$, for any $u\in (\calB + \calK^\circ)\cap \calK$, there exist $v\in\calB$ and $w\in \calK^\circ$ such that $u = v+w$. Now, take any $w'\in \calK^\circ$, we have $u+w' = (v+w) + w'=v+(w + w')\in \calB + \calK^\circ$ (since $\calK^\circ$ is a convex cone), and hence $((\calB + \calK^\circ)\cap \calK) + \calK^\circ\subseteq \calB + \calK^\circ$. 
\end{proof}

\begin{proof}[Proof of Theorem~\ref{thm:calC_subseteq}]
To show Part~\ref{item:conv_eq2_i}, first note that~\ref{item:conv_eq2_i_a}~$\Rightarrow$~\ref{item:conv_eq2_i_b} 
 straightforwardly follows from Theorem~\ref{thm:main} and Proposition~\ref{prop:conv_hull_eq2}\ref{item:conv_eq3_i}, and that~\ref{item:conv_eq2_i_b}~$\Rightarrow$~\ref{item:conv_eq2_i_c} is trivial.   To show~\ref{item:conv_eq2_i_c}~$\Rightarrow$~\ref{item:conv_eq2_i_a}, note that from Theorem~\ref{thm:feasible}, we have $ \lambda^{-1}(\conv (\calC\cap\calK)+\calK^\circ) = \conv\lambda^{-1}(\calC) = \lambda^{-1}(((\conv \calC)\cap\calK)+\calK^\circ)$. 
 Then by Lemma~\ref{lem:A_B}, we have $(\conv \calC)\cap\calK\subseteq \conv (\calC\cap\calK)+\calK^\circ$. 
 
 To show Part~\ref{item:conv_eq2_ii}, first note that~\ref{item:conv_eq2_ii_a}~$\Rightarrow$~\ref{item:conv_eq2_ii_b} 
 straightforwardly follows from Theorem~\ref{thm:main} and Proposition~\ref{prop:conv_hull_eq2}\ref{item:conv_eq3_ii}, and that~\ref{item:conv_eq2_ii_b}~$\Rightarrow$~\ref{item:conv_eq2_ii_c} is trivial.   To show~\ref{item:conv_eq2_ii_c}~$\Rightarrow$~\ref{item:conv_eq2_ii_a}, note that from Theorem~\ref{thm:feasible}, we have $\cl \lambda^{-1}(\clconv (\calC\cap\calK)+\calK^\circ) = \clconv\lambda^{-1}(\calC) = \cl\lambda^{-1}(((\clconv \calC)\cap\calK) +\calK^\circ)$, and hence for any $c\in \bbV$, we have 
 \begin{align}
 \begin{split}
  &\textstyle\sup_{u\in (\clconv (\calC\cap\calK)+\calK^\circ)\cap\calK}\; \ipt{\lambda(c)}{u}=\sup_{x\in \lambda^{-1}(\clconv (\calC\cap\calK)+\calK^\circ)}\;\ipt{c}{x}\\
&\textstyle\qquad\qquad  = \sup_{x\in \lambda^{-1}(((\clconv \calC)\cap\calK) +\calK^\circ)}\;\ipt{c}{x} = \sup_{u\in (((\clconv \calC)\cap\calK) +\calK^\circ)\cap\calK}\;\ipt{\lambda(c)}{u}.
\end{split} \label{eq:ipt_lambda_c_eq}
 \end{align}
Now, suppose that~\ref{item:conv_eq2_ii_a} fails to hold, then there exists  $u\in (\clconv\calC)\cap\calK$ but $u\not\in \cl(\clconv(\calC\cap\calK) + \calK^\circ)$. Then by strong separation, there exists $a\in \bbW$ such that 
\begin{equation}
\ipt{a}{u}>\textstyle\sup_{v\in \cl(\clconv(\calC\cap\calK) + \calK^\circ)}\; \ipt{a}{v} = \sup_{v_1\in \clconv(\calC\cap\calK) }\; \ipt{a}{v_1} + \sup_{v_2\in  \calK^\circ}\; \ipt{a}{v_2},  \label{eq:ipt_a_>}
\end{equation}
and hence $\sup_{v_2\in  \calK^\circ}\; \ipt{a}{v_2}<+\infty$. This implies that $a\in (\calK^\circ)^\circ = \calK$ (since $\calK$ is closed and convex). %
Since $u\in (\clconv\calC)\cap\calK$ and $a\in\calK$, by Lemma~\ref{lem:B}, we have  
 \begin{equation}
 \begin{split}
\ipt{a}{u}&\le \textstyle\sup_{v\in (\clconv\calC)\cap\calK}\; \ipt{a}{v} = \sup_{v\in ((\clconv\calC)\cap\calK) + \calK^\circ}\; \ipt{a}{v}\\
 &\textstyle = \sup_{v\in ((((\clconv\calC)\cap\calK) + \calK^\circ)\cap \calK) + \calK^\circ }\; \ipt{a}{v}= \sup_{v\in (((\clconv\calC)\cap\calK) + \calK^\circ)\cap \calK}\; \ipt{a}{v} . 
\end{split} \label{eq:ipt_a_le}
\end{equation}
 On the other hand, from the first inequality in~\eqref{eq:ipt_a_>}, we have 
 \begin{equation}
\ipt{a}{u}>\textstyle\sup_{v\in \clconv(\calC\cap\calK) + \calK^\circ}\; \ipt{a}{v}\ge \sup_{v\in (\clconv(\calC\cap\calK) + \calK^\circ)\cap\calK}\; \ipt{a}{v}. \label{eq:ipt_a_>2}
\end{equation}
By combining~\eqref{eq:ipt_a_le} and~\eqref{eq:ipt_a_>2} and letting $c\in \lambda^{-1}(a)$ in~\eqref{eq:ipt_lambda_c_eq}, we reach a contradiction. 
\end{proof}

%



 \subsection{Specializing the Results in Section~\ref{sec:second} to Invariant Sets} \label{sec:invariant}
 
Let $\calC\subseteq \bbW$ be invariant. From Definition~\ref{def:invariant}, there exists a reduced system $(\bbW, \bbW , \mu )$ of $(\bbV,\bbW,\lambda)$ such that $\calC$ is  {\em spectral} in $(\bbW, \bbW , \mu)$. In this section, we shall show that any invariant 
set $\calC$ actually 
satisfies the geometric assumptions in Theorem~\ref{thm:calC_subseteq} and Proposition~\ref{prop:conv_hull_eq2}, 
and as a result, 
we can utilize these results 
to obtain convexifications of $\lambda^{-1}(\calC)$ for any invariant set $\calC$. 

\begin{prop} \label{prop:C_spec_subseteq}
For all $u\in\bbW$, we have $[u]\subseteq \mu(u)+  \calK^\circ$. Consequently, for any nonempty spectral set $
\calC\subseteq \bbW$  in $(\bbW, \bbW , \mu)$,  we have $\calC \subseteq \calC\cap\calK + \calK^\circ$. 
\end{prop}

\noindent 
Before proving Proposition~\ref{prop:C_spec_subseteq}, note that if $\calC$ satisfies that $\calC \subseteq \calC\cap\calK + \calK^\circ$, then it certainly satisfies that $(\conv\calC)\cap\calK\subseteq \conv(\calC\cap\calK) + \calK^\circ$ and $(\clconv\calC)\cap\calK\subseteq \cl(\clconv(\calC\cap\calK) + \calK^\circ)$. 
In addition, note that any nonempty spectral set $\calC$  in $(\bbW, \bbW , \mu)$ is feasible (cf.~Remark~\ref{rmk:nonempty_spectral}). 
Therefore, we can specialize Theorem~\ref{thm:calC_subseteq} and Proposition~\ref{prop:conv_hull_eq2} to this case, and obtain 
the following two corollaries.

\begin{corollary}\label{cor:conv_K^o}
Let $\calC\subseteq \bbW$ be a nonempty spectral set in $(\bbW, \bbW , \mu)$. 
\begin{enumerate}[label=(\roman*),leftmargin=1.5\parindent]
\item For any convex $\calD$ that satisfies $\conv(\calC\cap\calK)\subseteq \calD \subseteq (\conv \calC)\cap\calK$, we have $$\conv\lambda^{-1}(\calC) = \lambda^{-1}(\calD+\calK^\circ).$$
\item For any convex $\calD$ that satisfies $\conv(\calC\cap\calK)\subseteq \calD \subseteq (\clconv \calC)\cap\calK$, we have $$\clconv\lambda^{-1}(\calC) = \cl\lambda^{-1}(\calD+\calK^\circ).$$ 
\end{enumerate}
\end{corollary}

 \begin{corollary}\label{cor:conv_hull_eq_inv}
Let $\calC\subseteq \bbW$ be a nonempty spectral set in $(\bbW, \bbW , \mu)$. 
\begin{enumerate}[label=(\roman*),leftmargin=1.5\parindent]
\item \label{item:conv_eq_inv_i} For any $\calD$ that satisfies $\calC\cap\calK\subseteq \calD \subseteq \conv \calC$, we have $\conv\lambda^{-1}(\calC) = \conv\lambda^{-1}(\calD).$ 
\item \label{item:conv_eq_inv_ii} For any $\calD$ that satisfies $\calC\cap\calK\subseteq \calD \subseteq \clconv \calC$, we have $\clconv\lambda^{-1}(\calC) = \clconv\lambda^{-1}(\calD).$
\end{enumerate}
\end{corollary}


\noindent 
In addition, we can easily obtain the following ``transfer principle'' regarding $\conv\lambda^{-1}(\calC)$. 

\begin{corollary}\label{cor:transfer}
Let $\calC\subseteq \bbW$ be a nonempty spectral set in $(\bbW, \bbW , \mu)$. 
Then $$\conv\lambda^{-1}(\calC) =\lambda^{-1}(\conv\calC) .$$ 
\end{corollary}

\begin{proof}
From~\cite[Theorem~3.2(a)]{Jeong_23}, we know that $\conv\lambda^{-1}(\calC) \subseteq \lambda^{-1}(\conv\calC) $. From Corollary~\ref{cor:conv_hull_eq_inv}\ref{item:conv_eq_inv_i}, we also have $\lambda^{-1}(\conv\calC) \subseteq  \conv\lambda^{-1}(\calC)$. 
\end{proof}

\begin{remark}\label{rmk:improve}
Note that for a spectral set $\calC$ in $(\bbW, \bbW , \mu)$, 
the result 
$\conv\lambda^{-1}(\calC) =\lambda^{-1}(\conv\calC)$ was proved in~\cite[Theorem~3.2(d)]{Jeong_23}, but under the additional compactness assumption of $\calC$. Without this assumption, only  $\conv\lambda^{-1}(\calC) \subseteq \lambda^{-1}(\conv\calC)$ was obtained in~\cite[Theorem~3.2(a)]{Jeong_23}. In view of this, Corollary~\ref{cor:transfer} indicates that the compactness assumption of $\calC$ is unnecessary, and the ``transfer principle'' always applies to $\conv\lambda^{-1}(\calC)$, as long as $\calC$ is spectral in $(\bbW, \bbW , \mu)$.  
\end{remark}

\begin{remark}
Note that the ``transfer principle'' regarding $\clconv\lambda^{-1}(\calC)$ was proved in~\cite[Theorem~3.2(b)]{Jeong_23}, namely $\clconv\lambda^{-1}(\calC) =\lambda^{-1}(\clconv\calC) $. 
As such, for a spectral set $\calC$ in $(\bbW, \bbW , \mu)$, we have two ways to characterize  $\conv\lambda^{-1}(\calC)$ and $\clconv\lambda^{-1}(\calC)$: one via Corollary~\ref{cor:conv_K^o}, and the other via the ``transfer principle''. 
\end{remark}


\noindent Now we turn to the proof of Proposition~\ref{prop:C_spec_subseteq}. To that end, we need the following lemma. 

\begin{lemma} \label{lem:mu_calC}
Let $
\calC\subseteq \bbW$ be a nonempty spectral set in $(\bbW, \bbW , \mu)$. Then $\mu(\calC) = \calC\cap\calK$.
\end{lemma}

\begin{proof}
For any $u\in \mu(\calC)$, there exists $v\in\calC$ such that $u = \mu(v)$. Since $\calC$ is spectral, we have $ [v]\subseteq \calC$ (cf.~Lemma \ref{lem:spectral_set}\ref{item:spectral_ii}) and by Remark~\ref{rmk:mu(u)_in_[u]}, we have  $\mu(v) \in [v]\subseteq \calC$. Since $\ran\mu=\calK$, we have $u=\mu(v) \in \calC\cap \calK$. Conversely, for any $u\in\calC\cap\calK$, we know that $\mu(u) = u$, and hence $u\in \mu(\calC)$. 
\end{proof}

\begin{proof}[Proof of Proposition~\ref{prop:C_spec_subseteq}]\renewcommand{\qedsymbol}{}
For any $v\in [u]$ and any $w\in\calK$, we have $\mu(v) = \mu(u)$ and 
\begin{equation}
\ipt{w}{v - \mu(u)} \le \ipt{w}{\mu(v)} - \ipt{w}{\mu(u)} = \ipt{w}{\mu(u)} - \ipt{w}{\mu(u)} = 0,
\end{equation}
and hence $v-\mu(u)\in \calK^\circ.$  For any spectral set $\calC$, by Lemma~\ref{lem:spectral_set}\ref{item:spectral_iii} and Lemma~\ref{lem:mu_calC}, we have 
\begin{equation}
\textstyle \calC = \bigcup_{u\in\calC}\, [u]\subseteq  \bigcup_{u\in\calC}\, (\mu(u) + \calK^\circ) = \bigcup_{v\in\mu(\calC)}\, (v + \calK^\circ) = \mu(\calC) + \calK^\circ= \calC\cap\calK + \calK^\circ.\tag*{$\square$} 
\end{equation}
\end{proof}

\vspace{-3ex}


\noindent
{\bf Acknowledgment.} The author sincerely thanks an anonymous reviewer for suggesting a simplified proof of Theorem~\ref{thm:main} during the initial submission of this paper. The author also thanks Casey Garner, M.\ S.\ Gowda, Bruno Lurenco,  Weijun Xie and Shuzhong Zhang for inspiring and helpful discussions during the preparation of this work.

\appendix
 
\section{Proof of Lemma~\ref{lem:spectral_set} }\label{app:proof}

To show \ref{item:spectral_i} $\Rightarrow$ \ref{item:spectral_ii}, let $\calE = \lambda^{-1}(\calC)$ for some feasible $\calC$. If $x\in \calE$, then $\lambda(x)\in\calC$, and hence $[x] = \lambda^{-1}(\lambda(x))\subseteq \lambda^{-1}(\calC) = \calE$. To show \ref{item:spectral_ii} $\Rightarrow$ \ref{item:spectral_iii}, it suffices to show $[\calE]\subseteq \calE$. If $x\in [\calE]$, then there exists $y\in\calE$ such that $x\in[y]$. Since $y\in\calE$, we have $[y]\subseteq \calE$ and hence $x\in\calE$. Note that \ref{item:spectral_iii} $\Rightarrow$ \ref{item:spectral_iv} is trivial. To show \ref{item:spectral_iv} $\Rightarrow$ \ref{item:spectral_i}, let $\calE = \bigcup_{u\in\calD}\,\lambda^{-1}(u)$ for some $\emptyset\ne \calD\subseteq \calK$. Then $\calE = \lambda^{-1}(\calD)$.

\bibliographystyle{IEEEtr}      
\bibliography{math_opt}


\end{document}